\documentclass[a4paper,oneside,reqno,11pt]{amsart}
\usepackage[latin1]{inputenc}
\usepackage[english]{babel}
\usepackage[margin=3.8cm]{geometry}
\usepackage{amsmath,amssymb,amsthm}
\usepackage{enumerate}
\usepackage{graphicx}
\usepackage{color}
\newtheorem{theorem}{Theorem}[section]

\newtheorem{lemma}[theorem]{Lemma}

\theoremstyle{definition}

\newcommand{\rr}{\mathbf{R}}

\newcommand{\ttt}{\mathbf{T}}
\newcommand{\eff}{{\mathrm{eff}}}

\newcommand{\ve}{\varepsilon}

\newcommand{\mop}[1]{\mathop{\mathrm{#1}}}

\title[Spectral asymptotics for a fourth order operator]{Spectral asymptotics for a singularly perturbed fourth order locally periodic self-adjoint elliptic operator}


\thanks{}

\makeatletter
\g@addto@macro{\endabstract}{\@setabstract}
\newcommand{\authorfootnotes}{\renewcommand\thefootnote{\@fnsymbol\c@footnote}}%
\makeatother

\begin{document}

\maketitle

\begin{center}

  \normalsize
  \authorfootnotes
  Alexandra Chechkina\footnote{The author was supported by the Norwegian Research Council under the Yggdrasil mobility program 2014.}\textsuperscript{1}, Iryna Pankratova\textsuperscript{2} and Klas Pettersson\textsuperscript{2}
  \par \bigskip

  \textsuperscript{1}Lomonosov Moscow State University, Russia \par
  \textsuperscript{2}Narvik University College, Norway\par
  \bigskip

\today
\end{center}

\begin{abstract}
We consider the homogenization of a singularly perturbed self-adjoint fourth order elliptic equation with locally periodic coefficients, stated in a bounded domain.
We impose Dirichlet boundary conditions on the boundary of the domain.
The presence of large parameters in the lower order terms and the dependence of the coefficients on the slow variable give rise to the effect of localization of the eigenfunctions.
We show that the $j$th eigenfunction can be approximated by a rescaled function that is constructed in terms of the $j$th eigenfunction of fourth or second order order effective operators with constant coefficients, depending on the large parameters.
\end{abstract}

\section{Introduction and problem statement}

We study the spectral asymptotics of a self-adjoint fourth order elliptic equation with locally periodic coefficients. The problem is stated in a bounded domain, and we impose Dirichlet boundary conditions on the boundary of the domain. The problem is a combination of homogenization and singular perturbation: because of the rapidly varying coefficients, homogenization arguments should be applied after a proper rescaling of the equation. As a result, we obtain an effective problem stated in the whole space, which will be of fourth or second order, depending on the choice of the large parameters ($\alpha, \beta$ in \eqref{eq:orig-prob}). We focus only on those cases when the localization of eigenfunctions is observed.

Similar problems for second order locally periodic elliptic operators, that are closely related to the present paper, were studied in \cite{ChPaPi-13, PaPe-2014}.
Dependence of the problem on a slow variable (in the coefficients or in the geometry) together with the presence of a large parameter in the equation give rise to the effect of localization of eigenfunctions. These results correspond to the so-called subcritical case, when eigenfunctions can be approximated by scaled exponentially decaying functions (eigenfunctions of a harmonic oscillator operator).  

A second order locally periodic elliptic operator with large potential was studied in \cite{AlPi-2002}.  
Homogenization of periodic elliptic systems with large potential was treated in \cite{AlCaPiSiVa-04}. In both cases, under a generic assumption on the ground state of an auxiliary cell problem, it was proved that the solution can be approximately factorized as the product of a fast oscillating cell eigenfunction and of a slowly varying solution of a scalar second-order equation. These two cases correspond to the so-called critical case.

There is a vast literature devoted to the homogenization of elliptic systems and higher order elliptic equations in domains with microstructure.For the homogenization of linear elliptic systems in the we refer to \cite{BLP-78, ZhOl-82, OlYoSh, BaPa-89}.
Homogenization of boundary value problems for higher order equations in domains with fine-grained  boundary were studied in \cite{Kh-72, Kh-77, MaKh-74, OlShap-96}. Homogenization of linear higher order equations in perforated domains were studied in  \cite{BeChu-83, Svish-92, Chao-97}; nonlinear higher order equations in perforated domains were considered in \cite{Koval-97, DalMasoSkrypnik-99}.
In \cite{Melnik-92} a spectral asymtotitics for a fourth order elliptic operator with rapidly oscillating coefficients was obtained. A spectral asymptotics for a biharmonis operator in a domain with a deeply indented boundary was constructed in \cite{KoNa-10}.

We turn to the formulation of the problem.
Let $\Omega$ be a bounded domain in $\mathbf{R}^d$ with Lipschitz boundary.
We consider the following Dirichlet spectral problem for a fourth order self-adjoint uniformly elliptic operator:
\begin{align}\label{eq:orig-prob}
\begin{cases}\displaystyle
\partial_{ij}\big(a_{ijkl}^\ve\partial_{kl} u^\ve\big)
- \frac{1}{\ve^{\alpha}}\partial_i \big(b_{ij}^\ve\partial_j u^\ve\big)
+ \frac{1}{\ve^\beta} c^\ve u^\ve = \lambda^\ve u^\ve, & x \in \Omega,\\
u^\ve = \nabla u^\ve \cdot n = 0, & x \in \partial \Omega,
\end{cases}
\end{align}
where $n$ denotes the exterior unit normal to $\partial \Omega$.
We use summation convention over repeated indices and use ''$\cdot$`` for the standard scalar product in $\mathbf R^d$;
$\ve >0$ is a small parameter; $\alpha, \beta$ are positive parameters.


Our main assumptions are:
\begin{enumerate}[(H1)]
\item The coefficients are real and of the form $a_{ijkl}^\ve(x) = a_{ijkl}\big(x, \frac{x}{\ve}\big)$, $b_{ij}^\ve(x) = b_{ij}\big(x, \frac{x}{\ve}\big)$ and $c^\ve(x) = c(x, \frac{x}{\ve})$,
where the functions $a_{ijkl}(x,y), b_{ij}(x, y)\in C(\overline{\Omega}; L^\infty({\ttt^d}))$, $c(x, y)\in C^{2}(\overline{\Omega}; L^\infty({\ttt^d}))$ are periodic in $y$; $\ttt^d$ is the unit torus.
We denote $|a_{ijkl}|, |b_{ij}| \le \Lambda^{-1}$, $\Lambda>0$.

\item
Symmetry condition: $a_{ijkl}=a_{klij}$, $b_{ij}=b_{ji}$.

\item The coefficients $a_{ijkl}(x, y)$ satisfy the uniform ellipticity condition in $\Omega \times \ttt^d$: there is $\Lambda > 0$ such that, almost everywhere,
\begin{align*}
a_{ijkl}(x, y) \xi_{ij}\xi_{kl}  \ge \Lambda |\xi|^2,\quad \forall \xi \in \mathbf R^{d\times d}.
\end{align*}

\item The function $c(x, y)$ is assumed to be strictly positive almost everywhere in $\Omega\times \ttt^d$, and its local average
\begin{align*}
\overline{c}(x) = \int_{\ttt^d} c(x, y)\, dy
\end{align*}
has a unique global minimum at $x=0 \in \Omega$, with a non-degenerate Hessian $H= \nabla \nabla \bar{c}(0)$:
\begin{align*}
\overline{c}(x) = \overline c(0) + \frac{1}{2}Hx\cdot x + o(|x|^2).
\end{align*}

\item The coefficients $b_{ij}(x,y)$ satisfy the uniform ellipticity condition in $\Omega \times \ttt^d$:
there is $\Lambda' > 0$ such that, almost everywhere,
\begin{align*}
b_{ij}(x,y)\xi_i \xi_j & \ge \Lambda' |\xi|^2, \quad \forall \xi \in \rr^d.
\end{align*}
\end{enumerate}


For domains $\mathcal{O}$ in $\rr^d$ we denote the $L^2(\mathcal{O})$-norm of $u$ by $\|u\|_{2, \mathcal{O}} = \int_{\mathcal{O}} |u|^2\, dx$, and for bounded Lipschitz domains $\Omega$,
\begin{align*}
H_0^2(\Omega) = \{v \in H^2(\Omega): \, v = \nabla v\cdot n = 0 \,\, \mbox{on} \,\, \partial \Omega\}.
\end{align*}

We consider the following bilinear form corresponding to \eqref{eq:orig-prob}
\begin{align}
\label{def:bilinear-form-original}
A_\ve(u, v) = \int_\Omega a_{ijkl}^\ve \partial_{ij} u \, \partial_{kl} v \, dx +
\frac{1}{\ve^{\alpha}} \int_\Omega b_{ij}^\ve \partial_{i} u \, \partial_{j} v \, dx
+ \frac{1}{\ve^\beta}\int_\Omega  c^\ve u  v \, dx.
\end{align}
The weak form of \eqref{eq:orig-prob} reads: Find $\lambda^\ve \in \mathbb C$ and nonzero $u^\ve \in H_0^2(\Omega)$ such that
\begin{align}
\label{eq:weak-orig-prob}
A_\ve(u^\ve, v) = \lambda^\ve \int_\Omega u^\ve v\, dx,
\end{align}
for all $v \in H_0^2(\Omega)$.

By the Riesz-Schauder, Hilbert-Schmidt theorems and the minmax principle (\cite{CouHi-53, reed1}), for each $\ve$ small enough, we have the following classical result.
\begin{lemma}
\label{lm:structure-spectrum-orig-prob}
Suppose that (H1)--(H4) are satisfied.
Then for all sufficiently small $\ve > 0$,
the eigenvalues $\lambda^\ve_i$ of \eqref{eq:orig-prob} are real and such that
\begin{align*}
0 < \lambda_1^\ve \le \lambda_2^\ve \le \cdots ,  \quad \lambda_i^\ve \to \infty \text{ as }  i & \to \infty,
\end{align*}
where each eigenvalue is counted as many times as its multiplicity.
The eigenfunctions $u_i^\ve$ form an orthonormal basis in $L^2(\Omega)$.
All eigenvalues are of finite multiplicity and are characterized by the variational principle:
\begin{align*}
\lambda_i^\ve & = \min \frac{A_\ve( v, v )}{\int_\Omega v^2 \,dx},
\end{align*}
where the minimum is taken over all nonzero functions $v$ in $H^2_0(\Omega)$ that are orthogonal in $L^2(\Omega)$
to the first $i-1$ eigenfunctions $u^\ve_1, \ldots, u^\ve_{i-1}$.
\end{lemma}

The goal of this paper is to describe the asymptotic behavior of the eigenpairs $(\lambda_k^\ve, u_k^\ve)$, as $\ve \to 0$.
We restrict ourselves to the values of the parameters $\alpha\ge 0, \beta > 0$ (singular perturbation) such that $\alpha < \beta$ (concentration effect is observed)  and $\beta < 4$ (subcritical case).
The result is presented in the three theorems \ref{th:main}, \ref{th:case R_1}, \ref{th:case R_3-R_5}.

The case $\alpha=\beta$ is classical, and the standard two-scale convergence can be applied to describe the asymptotics. Depending on the value of $\alpha$ one gets either fourth order limit operator without second order terms ($0<\alpha=\beta<2$), or fourth order with second order terms ($\alpha=\beta=0$, $\alpha=\beta=2$), or just a second order limit operator ($\alpha=\beta>2$).

The case $\beta=4, \alpha<4$ is the critical case, when the oscillations of the eigenfunctions are expected to be of order $\ve$. As it is seen from \cite{AlCaPiSiVa-04}, the technique to be used is different, and this case is to be considered elsewhere. In addition, the values of $\alpha, \beta$ such that $3\le \beta< 4$ and $\alpha < \beta-2$ are not covered by the present paper (the hatched region in Figure~\ref{fig:alphabeta}), because the error coming from Lemma~\ref{lm:mvt} while passing to the limit does not vanish, as $\ve \to 0$. Another argument is to be applied, and this case will be considered elsewhere.

To describe the asymptotic behavior of eigenpairs $(\lambda_k^\ve, u_k^\ve)$ as $\ve \to 0$, we divide the domain for the parameters
$(\alpha, \beta)$ into the following regions (see Figure \ref{fig:alphabeta}):
\begin{align*}
R_1 & = \{ (\alpha,\beta) : 0 \le \alpha < 1, \,\, 3\alpha < \beta <3\}, \\
R_2 & = \{ (\alpha, \beta) : 0 < \alpha < 1, \,\, \beta=3\alpha\}, \\
R_3 & = \{ (\alpha, \beta) : 0 < \alpha < 2, \,\, \alpha<\beta<3\alpha, \,\, \beta < \alpha + 2\}, \\
R_4 & = \{ (\alpha, \beta) : \alpha=2, \,\, 2 < \beta < 4\}, \\
R_5 & = \{ (\alpha, \beta) : 2 < \alpha < 4, \,\, \alpha<\beta<4 \}.
\end{align*}
The reason for distinguishing these regions is that we get different asymptotics in each case.
In short, in $R_1$ we get a fourth order equation in the limit without second order terms;
in $R_2$ the limit equation contains both fourth and second order terms;
in $R_3, R_4, R_5$ the limit equations are of the second order.
We choose to consider in details one case $\beta=3\alpha=1$, corresponding to region $R_2$, since all the terms contribute in the limit (see Theorem~\ref{th:main}).
The spectral asymptotics in the other cases are described in Sections~\ref{sec:R_1}, \ref{sec:R_3-R_5}.

\begin{figure}[hb]
\begin{center}
\includegraphics[width=.45\textwidth]{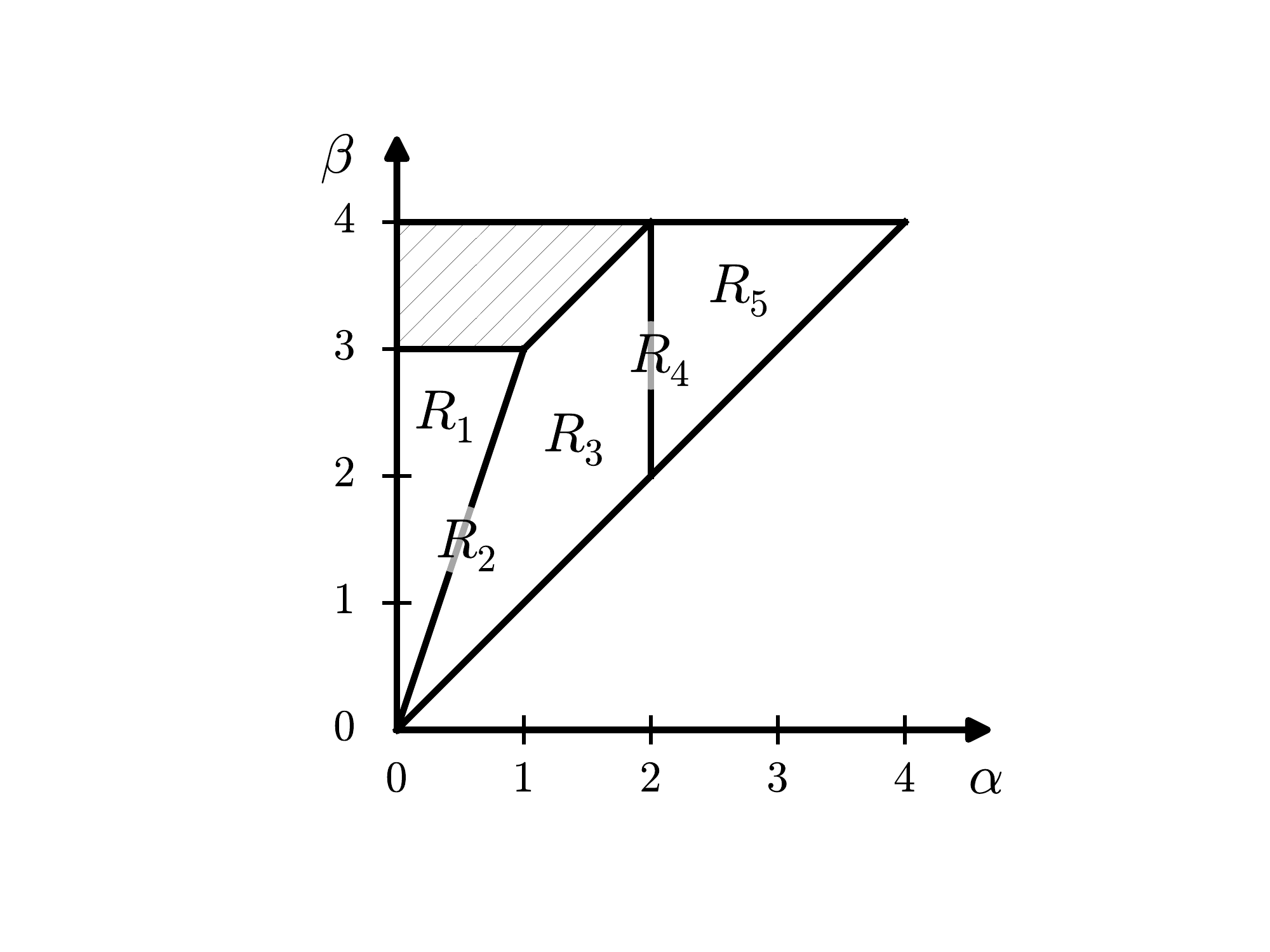}
\end{center}
\caption{The partition of the parameter region for $(\alpha, \beta)$.}
\label{fig:alphabeta}
\end{figure}

\section{A model problem: the case $(\alpha, \beta) \in R_2$}
\label{sec:modelproblem}

The result of this section is contained in the following theorem.
\begin{theorem}
\label{th:main}
Let $(\alpha, \beta)\in R_2$ and let $(\lambda_k^\ve, u_k^\ve)$ be the $k$th eigenpair of \eqref{eq:orig-prob} normalized by $\|u_k^\ve\|_{2, \Omega}^2=\ve^{d\alpha/2}$. Suppose that the conditions (H1)--(H4) are satisfied.
Then we have the following representation:
\begin{align*}
\lambda_k^\ve & = \frac{\bar c(0)}{\ve^{3\alpha}} + \frac{\eta_k^\ve}{\ve^{2\alpha}}, &
u_k^\ve(x) & = v_k^\ve\Big(\frac{x}{\ve^{\alpha/2}}\Big),
\end{align*}
where $(\eta_k^\ve, v_k^\ve)$ are such that as $\ve \to 0$,
\begin{enumerate}[(i)]
\item $\eta_k^\ve \to \eta_k$,
\item up to a subsequence, $v_k^\ve$ converges to $v_k$ weakly in $H^2(\rr^d)$ and strongly in $L^2(\rr^d)$,
\end{enumerate}
where $\eta_k$ is the $k$th eigenvalue, and $v_k$ is an eigenfunction corresponding to $\eta_k$ normalized by $\| v_k \|_{2,\rr^d} = 1$, of the uniformly elliptic effective problem
\begin{align*}
\partial_{ij}\big(a_{ijkl}^\eff \partial_{kl} v\big)
- \partial_i \big(\bar b_{ij}(0)\partial_j v\big)
+ \frac{1}{2} (H z\cdot z)\, v = \eta v, \quad z \in \rr^d, 
\end{align*}
with $a_{ijkl}^\eff$ defined by \eqref{def:eff coefficients}, $\bar b_{ij}(0)=\int_{\ttt^d} b_{ij}(0, y)dy$, and $H = \nabla\nabla \bar c(0)$.
\end{theorem}

The proof of Theorem~\ref{th:main} will occupy the rest of this section and is given for the case $\beta=3\alpha=1$.
The argument used is the same for the other values of $(\alpha, \beta)\in R_2$.

\subsection{Estimates for eigenvalues of the original problem}
\label{Sec:estimates eigenvalues}

To motivate the change of variables we will make in the next subsection we prove
the following a priori estimates for the eigenvalues and the eigenfunctions of problem \eqref{eq:orig-prob}.

\begin{lemma}\label{lm:est-lambda^eps}
Suppose that (H1)--(H4) are satisfied.
Let $(\lambda_i^\ve, u_i^\ve)$ be the $i$th eigenpair of \eqref{eq:orig-prob} with $\beta=3\alpha=1$, normalized by $\| u_i^\ve \|_{2,\Omega} = 1$.
Then there exist positive constants $C_1, C_2(i)$ such that for all sufficiently small $\ve > 0$,
\begin{align*}
-\frac{C_1}{\ve^{2/3}} & \le \lambda_i^\ve - \frac{\bar c(0)}{\ve} \le \frac{C_2(i)}{\ve^{2/3}}, &
\| \Delta u_i^\ve\|_{2, \Omega} & \le \frac{C}{\ve^{1/3}}.
\end{align*}
\end{lemma}

To prove Lemma \ref{lm:est-lambda^eps} we will use the following estimate for integrals of oscillating functions.
\begin{lemma}\label{lm:mvt}
Let $g(x,y) \in C^2(\overline{\Omega}; L^\infty(\ttt^d))$ be such that $\int_{\ttt^d} g(x,y)\,dy = 0$ for all $x \in \overline{\Omega}$.
Then there exists a positive constant $C$ such that
\begin{align*}
\Big|  \int_\Omega g\big(x,\frac{x}{\ve}\big) u v \,dx \Big| \le C \ve^2 ( \| \Delta u \|_{2, \Omega} \| v \|_{2, \Omega} + \| u \|_{2, \Omega} \| \Delta v \|_{2, \Omega}),
\end{align*}
for all $u, v \in H^2_0(\Omega)$.
\end{lemma}
\begin{proof}
Let $\Psi(x,y) \in C^2(\overline{\Omega}; C^1({\ttt^d}))$, periodic in $y$, be defined by
\begin{align*}
\Delta_y \Psi(x,y) & = g(x,y), \,\,\,\, y \in \ttt^d,
\end{align*}
Since the local average of $g(x,y)$ is zero, $\Psi$ is well-defined.
By the Green formula,
\begin{align*}
& \int_\Omega g\big(x,\frac{x}{\ve}\big) u(x) v(x) \,dx
= \int_\Omega (\Delta_y \Psi)\big(x,\frac{x}{\ve}\big) u(x) v(x) \,dx \\
& \quad = \ve^2 \int_\Omega \big(  \Delta \Psi\big(x,\frac{x}{\ve}\big) - 2\mop{div}( (\nabla_x \Psi)(x,\frac{x}{\ve}) ) + (\Delta_x \Psi)(x,\frac{x}{\ve}) \big) u(x)v(x) \,dx \\
& \quad = \ve^2 \int_\Omega \big(  \Psi\big(x,\frac{x}{\ve}\big)  \Delta (uv)  + 2 (\nabla_x \Psi)(x,\frac{x}{\ve})  \cdot \nabla (uv) + (\Delta_x \Psi)(x,\frac{x}{\ve}) uv \big)  \,dx.
\end{align*}
After an application of the Green formula to the $\nabla u \cdot \nabla v$ term coming from $\Delta(u v)$, by the H{\"o}lder and triangle inequalities, we have
\begin{align*}
& \Big|  \int_\Omega g\big(x,\frac{x}{\ve}\big) u(x) v(x) \,dx \Big| \\
& \quad \le \ve^2 \Big( \| \Psi \|_{L^\infty(\Omega \times \ttt^d)} ( \| \Delta u \|_{2, \Omega} \| v \|_{2, \Omega} + \| u \|_{2, \Omega} \| \Delta v \|_{2, \Omega} ) \\
& \qquad\qquad + 2\| \nabla_x \Psi \|_{L^\infty(\Omega \times \ttt^d)} ( \| \nabla u \|_{2, \Omega} \| v \|_{2, \Omega} + \| u \|_{2, \Omega} \| \nabla v \|_{2, \Omega} ) \\
& \qquad\qquad + \| \Delta_x \Psi \|_{L^\infty(\Omega \times \ttt^d)} \| u \|_{2, \Omega} \| v \|_{2, \Omega} \Big).
\end{align*}
The estimate follows from the regularity of $\Psi$ and the Poincar\'e inequality.
\end{proof}

\begin{proof}[Proof of Lemma \ref{lm:est-lambda^eps}]
Let $v \in C^\infty_0(\rr^d)$ be such that $\| v \|_{2,\rr^d} = 1$, and set $v^\ve(x) = v(\ve^{-1/6}x)$.
We assume that $\ve$ is small enough such that $\mop{supp} v^\ve \subset \Omega$.
By the variational principle,
\begin{align}\label{eq:vp}
\lambda_1^\ve
& \le
\frac{\int_\Omega a_{ijkl}^\ve \partial_{ij} v^\ve \partial_{kl}^2 v^\ve \, dx +
\ve^{-1/3} \int_\Omega b_{ij}^\ve \partial_{i} v^\ve \partial_{j} v^\ve \, dx}{\int_\Omega (v^\ve)^2 \,dx}
+ \frac{1}{\ve} \frac{\int_{\Omega} c^\ve(v^\ve)^2 \,dx}{\int_\Omega (v^\ve)^2 \,dx}.
\end{align}
By the boundedness of the coefficients, the first fraction in \eqref{eq:vp} is bounded by $C \ve^{-2/3}$.
The second fraction in \eqref{eq:vp} is estimated using (H4) and Lemma \ref{lm:mvt}:
\begin{align*}
\frac{1}{\ve} \frac{\int_{\Omega} c(x,\frac{x}{\ve})(v^\ve)^2 \,dx}{\int_\Omega (v^\ve)^2 \,dx}
& = \frac{\bar c(0)}{\ve} + \frac{1}{\ve} \int_{\rr^d} ( \bar c(\ve^{1/6}x) -  \bar c( 0 )) v^2 \,dx \\
& \quad + \frac{1}{\ve} \int_{\rr^d} ( c(\ve^{1/6}x, \frac{x}{\ve^{5/6}}) -  \bar c( \ve^{1/6}x )) v^2 \,dx \\
& \le \frac{\bar c(0)}{\ve} +  C \ve^{ -2/3 },
\end{align*}
for sufficiently small $\ve > 0$.

In order to obtain an estimate from below for the first eigenvalue $\lambda_1^\ve$, we need to estimate the second derivatives of the corresponding eigenfunctions.
Let $u_1^\ve$ denote any eigenfunction corresponding to $\lambda_1^\ve$, normalized by $\|u_1^\ve\|_{2,\Omega} = 1$. Then, by \eqref{eq:weak-orig-prob},
\[
\lambda_1^\ve - \frac{1}{\ve}\int_\Omega c(x,\frac{x}{\ve}) (u_1^\ve)^2 \,dx
= \int_\Omega a_{ijkl}^\ve \partial_{ij} u_1^\ve \partial_{kl} u_1^\ve \, dx +
\frac{1}{\ve^{1/3}} \int_\Omega b_{ij}^\ve \partial_{i} u_1^\ve \partial_{j} u_1^\ve \, dx.
\]
On the one hand, by the ellipticity of $a_{ijkl}$ and the boundedness of the coefficients,
\begin{align}
\label{eq:frombelow}
\lambda_1^\ve - \frac{1}{\ve}\int_\Omega c(x,\frac{x}{\ve}) (u_1^\ve)^2 \,dx  
& \ge \Lambda \| \nabla\nabla u_1^\ve \|_{2, \Omega}^2 - C \frac{1}{\ve^{1/3}}\int_\Omega |\nabla u_1^\ve|^2\,dx \notag\\
& \ge \Lambda \| \nabla\nabla u_1^\ve \|_{2, \Omega}^2  - C_2\gamma \| \Delta u_1^\ve\|_{2, \Omega}^2 - \frac{C_2}{4\gamma}\ve^{-2/3}\| u_1^\ve \|_{2, \Omega}^2 \notag\\
& \ge C ( \| \nabla\nabla u_1^\ve \|_{2, \Omega}^2 - \ve^{-2/3} ),
\end{align}
where $\gamma > 0$ in the Cauchy inequality is chosen small enough such that the resulting constant $C$ is positive. Note that one can choose $\gamma$ that depends just on the ellipticity constant of $a_{ijkl}$ and the upper bound for $b_{ij}$.

On the other hand, from the upper estimate for $\lambda^\ve_1$,
\begin{align}
& \lambda_1^\ve - \frac{1}{\ve}\int_\Omega c(x,\frac{x}{\ve}) (u_1^\ve)^2 \,dx \notag\\
& \quad \le \frac{\bar c(0)}{\ve} - \frac{1}{\ve}\int_\Omega c(x,\frac{x}{\ve})(u_1^\ve)^2 \,dx + C \ve^{-2/3} \notag\\
& \quad = \frac{1}{\ve}\int_{\Omega} ( \bar c(0) - \bar c(x) ) (u_1^\ve)^2 \,dx
+ \frac{1}{\ve}\int_{\Omega} ( \bar c(x) - c(x,\frac{x}{\ve}) ) (u_1^\ve)^2 \,dx + C\ve^{-2/3} \notag\\
& \quad \le C (\ve \| \Delta u_1^\ve \|_{2,\Omega} + \ve^{-2/3}), \label{eq:fromabove}
\end{align}
where we in the third step used that $0$ is a minimum point for $\bar c(x)$ by (H4), and Lemma \ref{lm:mvt}.

Combining \eqref{eq:frombelow} and \eqref{eq:fromabove}, by the Cauchy inequality,
we deduce that,
\begin{align}\label{eq:uest}
\| \nabla\nabla u_1^\ve \|^2_{2,\Omega} \le C \ve^{-2/3}.
\end{align}

We proceed with the estimate from below for $\lambda_1^\ve$.
By \eqref{eq:frombelow} and \eqref{eq:uest} we have
\begin{align*}
\lambda_1^\ve
& \ge \frac{1}{\ve}\int_\Omega c(x,\frac{x}{\ve})(u_1^\ve)^2 \,dx - \frac{C}{\ve^{2/3}} \\
& =  \frac{\bar c(0)}{\ve}
+ \frac{1}{\ve}\int_\Omega (\bar c(x) - \bar c(0))(u_1^\ve)^2 \,dx
+ \frac{1}{\ve}\int_\Omega (c(x,\frac{x}{\ve}) - \bar c(x))(u_1^\ve)^2 \,dx
- \frac{C}{\ve^{2/3}}.
\end{align*}
By (H4) and Lemma~\ref{lm:mvt}, combined with \eqref{eq:uest}
\begin{align*}
\lambda_1^\ve
\ge \frac{\bar c(0)}{\ve} - \frac{C}{\ve^{2/3}},
\end{align*}
for all sufficiently small $\ve > 0$. In this way we have obtained the required estimates for the first eigenvalue $\lambda_1^\ve$. Since $\lambda_1^\ve$ is the smallest eigenvalue, the estimate from below for $\lambda_i^\ve$, $i>1$, follows from the corresponding estimate for $\lambda_1^\ve$.

To estimate $\lambda^\ve_i > \lambda_1^\ve$ for $i > 1$, one can use as a test function the projection of $v^\ve$ onto the orthogonal complement of the span of the first $i - 1$ eigenvectors, with respect to the $L^2(\Omega)$ inner product.
Since the span is finite dimensional this projection is nonzero for all sufficiently small $\ve > 0$.

Let $m_1$ be the multiplicity of the first eigenvalue $\lambda_1^\ve = \lambda_2^\ve =\cdots = \lambda_{m_1}^\ve$. 
We estimate from above $\lambda_{m_1}^\ve$, similar arguments can be applied to estimate other eigenvalues.

For $v\in C_0^\infty(\rr^d)$, we introduce $v^\ve(x) = v(\frac{x}{\ve^{1/6}})$ and denote
$\pi_{\ve, k} = \int_{\Omega} v^\ve u_k^\ve \,dx$, $k=1, \ldots, m_1$. Then $V^\ve(x) = v^\ve(x) - \sum_k \pi_{\ve, k}\, u_k^\ve(x)$ is orthogonal in $L^2(\Omega)$ to $\mop{span}\{u_1^\ve, \ldots, u_{m_1}^\ve\}$. For convenience we assume the normalization condition
\begin{align}
\label{eq:ax-2}
\|V^\ve\|_{2, \Omega}^2 = \|v^\ve\|_{2, \Omega}^2 - \sum_{k=1}^{m_1} \pi_{\ve, k}^2 = \ve^{d/6}.
\end{align}
Using $V^\ve$ as a test function in the variational principle, we deduce that
\begin{align}
\label{eq:ax-1}
\lambda_{m_1+1}^\ve \le \ve^{-d/6}\big(A_\ve(v^\ve, v^\ve) - \lambda_1^\ve \sum_k \pi_{\ve,k}^2\big).
\end{align}
By (H1)--(H3) and Lemma~\ref{lm:mvt} we get
\begin{align*}
A_\ve(v^\ve, v^\ve)- \lambda_1^\ve \sum_k \pi_{\ve,k}^2
& \le \frac{\bar c(0)}{\ve}\|v^\ve\|_{2, \Omega}^2 - \lambda_1^\ve \sum_k \pi_{\ve,k}^2
+ C \ve^{d/6}\ve^{-2/3}.
\end{align*}
Due to \eqref{eq:ax-2}, \eqref{eq:ax-1} and the estimate from below for $\lambda_1^\ve$,
\begin{align*}
\lambda_{m_1+1}^\ve \le \frac{\bar c(0)}{\ve} + \frac{C_1 \sum_k \pi_{\ve,k}^2}{\ve^{d/6}\ve^{2/3}} + \frac{C}{\ve^{2/3}}.
\end{align*}
Due to the normalization condition for $u_1^\ve$,
\[
\sum_k\pi_{\ve, k}^2 \le m_1 \ve^{d/6}\|v\|_{2, \rr^d}^2,
\]
thus
\begin{align*}
\lambda_{m_1+1}^\ve \le \frac{\bar c(0)}{\ve} + \frac{C_2}{\ve^{2/3}},
\end{align*}
and the estimate is proved.
\end{proof}

\subsection{Rescaling the problem and computing the asymptotics}
\label{Sec:rescaled problem}

\medskip
Led by Lemma \ref{lm:est-lambda^eps}, we shift the spectrum of \eqref{eq:orig-prob} by $\bar c(0)/\ve$ and rescale
such that the eigenvalues become bounded.
Let
\begin{align}
\label{change of variables}
z & = \frac{x}{\ve^{1/6}}, &
v^\ve(z) & = u^\ve(\ve^{1/6}z), &
\eta^\ve & = \ve^{2/3}\big(\lambda^\ve - \frac{\bar c(0)}{\ve}\big), &
\Omega_\ve & = \ve^{-1/6}\Omega.
\end{align}
Then \eqref{eq:weak-orig-prob} takes the form
\begin{align}
\label{eq:weak-rescaled-prob}
& \int_{\Omega_\ve} \hat a_{ijkl}^\ve \partial_{ij} v^\ve \partial_{kl} \varphi \, dz +
\int_{\Omega_\ve} \hat b_{ij}^\ve \partial_{i} v^\ve \partial_{j} \varphi \, dz
+ \frac{1}{\ve^{1/3}}\int_{\Omega_\ve}  (\hat c^\ve - \bar c(0)) v^\ve \varphi \, dz \notag \\
& \quad = \eta^\ve \int_{\Omega_\ve} v^\ve \varphi \, dz,
\end{align}
for any $v \in H_0^2(\Omega_\ve)$, where
\begin{align*}
\hat a_{ijkl}^\ve(z) & = a\big(\ve^{1/6}z, \frac{z}{\ve^{5/6}}\big), &
\hat b_{ij}^\ve(z) & = b\big(\ve^{1/6}z, \frac{z}{\ve^{5/6}}\big), &
\hat c^\ve(z) & = c\big(\ve^{1/6}z, \frac{z}{\ve^{5/6}}\big).
\end{align*}

For the rest of $R_2$, one uses $z = \ve^{-\alpha/2}x$.

In order to describe the asymptotic behavior of eigenpairs $(\eta_k^\ve, v_k^\ve)$ of \eqref{eq:weak-rescaled-prob},
we consider the Green operator
\begin{align*}
G_\ve: \,\, L^2(\Omega_\ve) & \,\, \to \,\,   L^2(\Omega_\ve), \\
 f_\ve & \,\, \mapsto \,\,  V^\ve,
\end{align*}
where $V^\ve \in H_0^2(\Omega_\ve)$ is the unique solution to the boundary-value problem
\begin{align}\label{eq:Green operator rescaled}
\begin{cases}
\partial_{ij}\big(\hat a_{ijkl}^\ve\partial_{kl} V^\ve\big)
+ \partial_i \big(\hat b_{ij}^\ve\partial_j V^\ve\big)
+ \frac{\hat c^\ve -\bar c(0)}{\ve^{1/3}} V^\ve  + \mu V^\ve = f_\ve, &  z \in \Omega_\ve,\\
V^\ve = \nabla V^\ve \cdot n = 0, &  z \in \partial \Omega_\ve.
\end{cases}
\end{align}
Here $\mu>0$ is a large enough constant, but depending just the ellipticity constant $\Lambda$. The operator $G_\ve$ can be considered as an operator from $L^2(\rr^d)$ into itself by extending the corresponding solution $V^\ve$ by zero outside $\Omega_\ve$. The existence and uniqueness is ensured by the Riesz-Fr\'echet representation theorem since the corresponding symmetric quadratic form
\begin{align}
\label{eq:quadratic form Green rescaled}
A_\ve(V^\ve, V^\ve) & = \int_{\Omega_\ve} \hat a_{ijkl}^\ve\partial_{kl} V^\ve \partial_{ij} V^\ve\, dz
 + \int_{\Omega_\ve} \hat b_{ij}^\ve\partial_j V^\ve \partial_i V^\ve\,dz \\
&\quad + \frac{1}{\ve^{1/3}} \int_{\Omega_\ve} (\hat c^\ve -\bar c(0)) (V^\ve)^2\, dz
 + \mu \int_{\Omega_\ve} (V^\ve)^2\, dz \notag
\end{align}
is coercive. Indeed, by (H1)--(H4) and Lemma~\ref{lm:mvt} we have:
\begin{align}
\label{eq:ax-7}
A_\ve(V^\ve, V^\ve) & \ge \Lambda \int_{\Omega_\ve} |\nabla\nabla V^\ve|^2\, dz
- \Lambda^{-1}\int_{\Omega_\ve} |\nabla V^\ve|^2\,dz \\
& \quad -  C_0\ve^{4/3} \|V^\ve\|_{2, \Omega_\ve} \|\Delta V^\ve\|_{2, \Omega_\ve} \notag\\
& \quad + \frac{1}{\ve^{1/3}}\int_{\Omega_\ve} (\bar c(\ve^{1/6}z) -\bar c(0)) (V^\ve)^2\, dz
 + \mu \|V^\ve\|_{2, \Omega_\ve}^2.\notag
\end{align}
Since the Hessian matrix $H$ is positive definite, there exist a positive constant $K_1$ such that
\begin{align*}
\bar c(\ve^{1/6}z)-\bar c(0) \ge K_1 |\ve^{1/6}z|^2.
\end{align*}
Due to the Dirichlet boundary conditions,
\begin{align}
\label{eq:ax-6}
\|\nabla V^\ve\|_{2, \Omega_\ve}^2 \le \|V^\ve\|_{2, \Omega_\ve}\|\Delta V^\ve\|_{2, \Omega_\ve}.
\end{align}
Applying the Cauchy inequality $\|V^\ve\|\, \|\Delta V^\ve\| \le \delta \|V^\ve\|^2 + \|\Delta V^\ve\|^2/\delta$ with $\delta =\Lambda/3$ in \eqref{eq:ax-7} we get 
\begin{align*}
A_\ve(V^\ve, V^\ve) & \ge \frac{\Lambda}{3} \int_{\Omega_\ve} |\nabla\nabla V^\ve|^2\, dz
+ K_1 \int_{\Omega_\ve} |z|^2 (V^\ve)^2\, dz\\
& \quad + \big(\mu - \frac{3}{\Lambda^3} - \frac{2C_0\ve^{4/3}}{\Lambda}\big) \|V^\ve\|_{2, \Omega_\ve}^2.
\end{align*}
For $\ve$ small enough, we can choose $\mu$ depending just on $\Lambda$ and $|\Omega|$ such that, for some positive constant $\tilde C$, we have 
\begin{align}
\label{eq:ax-8}
&A_\ve(V^\ve, V^\ve) \ge \tilde C \big(\|\nabla\nabla V^\ve\|_{2, \Omega_\ve}^2
+ \int_{\Omega_\ve} |z|^2 (V^\ve)^2\, dz
 + \|V^\ve\|_{2, \Omega_\ve}^2\big).
\end{align}
Even though $A_\ve(V^\ve, V^\ve) \ge C\| \Delta V^\ve \|_{2,\Omega_\ve}$ is enough for coercivity
of the quadratic form, we will make use of the last inequality.
The addition of the constant $\mu$ has the effect of shifting the entire spectrum
of \eqref{eq:weak-rescaled-prob} by $\mu$.

We  introduce also the limit Green operator
\begin{align*}
G: \,\, L^2(\rr^d) & \,\, \to \,\,   L^2(\rr^d), \\
 f & \,\, \mapsto \,\,  V,
\end{align*}
where $V\in H^2(\rr^d)$ is the unique solution to the equation
\begin{align}
\label{eq:eff Green operator}
\partial_{ij}\big(a_{ijkl}^\eff \partial_{kl} V\big)
- \partial_i \big(\bar b_{ij}(0)\partial_j V\big)
+ \frac{1}{2} (H z\cdot z)\, V + \mu V = f, \quad z \in \rr^d. 
\end{align}
Here $H$ is the Hessian matrix of $\bar c$ at $x=0$ (see (H4)), and the effective coefficients are defined by
\begin{align}
a_{ijkl}^\eff & = \int_{\ttt^d} (a_{ijmn}(0, y)\partial_{mn}N_{kl}(y) + a_{ijkl}(0, y))\, dy, \label{def:eff coefficients} \\
\bar b_{ij}(x) & = \int_{\ttt^d} b_{ij}(x, y) \, dy,\notag
\end{align}
where the periodic functions $N_{kl} \in H^2(\ttt^d)/\rr$ solve the following cell problems:
\begin{align}
\label{eq:N}
\partial_{ij}(a_{ijmn}(0, y)\partial_{mn}N_{kl}(y)) = -\partial_{ij}a_{ijkl}(0, y), \quad y \in \ttt^d.
\end{align}
Due to the periodicity of $a_{ijkl}$ in $y$, the above problem is well-posed, the solution $N_{kl}$ is unique up to an additive constant.

\begin{lemma}\label{lm:aeffcoercive}
Under the assumptions (H1)--(H4),
$a^\eff$ defined by \eqref{def:eff coefficients} is coercive on $\rr^{d \times d}$, i.e. there is a positive constant $C$ such that
$a^\eff_{ijkl}\xi_{ij}\xi_{kl} \ge C|\xi|^2$ for all $\xi \in \rr^{d \times d}$.
\end{lemma}
\begin{proof}
Using (H1)--(H4) gives a well-defined $a^\eff$.
We rewrite \eqref{def:eff coefficients} as
\begin{align*}
a^\eff_{ijkl} & = \int_{\ttt^d} \delta_{pi} \delta_{qj} a_{pqrs}(0,y) ( \delta_{rk}\delta_{sl} + \partial_{rs}N_{kl} ) \,dy.
\end{align*}
Using $N_{ij}$ as a test function in equation \eqref{eq:N} for $N_{kl}$ gives
\begin{align*}
\int_{\ttt^d} a_{pqrs}(0,y)( \delta_{rk}\delta_{sl} + \partial_{rs}N_{kl} ) \partial_{pq}N_{ij} \,dy & = 0.
\end{align*}
Thus
\begin{align*}
a^\eff_{ijkl} & = \int_{\ttt^d} a_{pqrs}(0,y) ( \delta_{rk}\delta_{sl} + \delta_{rs}N_{kl} ) (\delta_{pi} \delta_{qj} + \partial_{pq}N_{ij}) \,dy.
\end{align*}
The last equation shows that $a^\eff$ is symmetric by (H2): $a^\eff_{ijkl} = a^\eff_{klij}$.
Moreover, with $\xi \in \rr^{d \times d}$ we obtain
\begin{align*}
a^\eff_{ijkl}\xi_{ij}\xi_{kl} & = \int_{\ttt^d} a_{pqrs}(0,y) ( \xi_{kl} (\delta_{rk}\delta_{sl} + \delta_{rs}N_{kl}) ) ( \xi_{ij}(\delta_{pi} \delta_{qj} + \partial_{pq}N_{ij})) \,dy \\
& \ge \Lambda \sum_{p,q}\int_{\ttt^d} \big| \xi_{ij}(\delta_{pi} \delta_{qj} + \partial_{pq}N_{ij}) \big|^2 \,dy,
\end{align*}
by the coerciveness of $a(0,y)$ guaranteed by (H3).
The last inequality implies that $a^\eff$ is positive definite.
Indeed, if $\xi \in \rr^{d \times d}$ is such that $a^\eff_{ijkl}\xi_{ij}\xi_{kl} = 0$, then
$\big| \xi_{ij}( \delta_{pi}\delta_{qj} + \partial_{pq}N_{ij} ) \big| = 0$ for all $p,q$.
It follows that $\nabla \xi_{ij}( y_i\delta_{qj} + \partial_{q}N_{ij} ) = 0$ for all $q$.
Therefore, $\xi_{ij}( y_i\delta_{qj} + \partial_{q}N_{ij} )$ is constant and so necessarily $\xi_{iq}y_i$ is periodic by the
periodicity of $N_{ij}$.
Hence $\xi_{iq} = 0$ for all $i, q$.
We conclude that for all $\xi \in \rr^{d \times d}\setminus \{0\}$,
\begin{align*}
a^\eff_{ijkl}\xi_{ij}\xi_{kl} & > 0.
\end{align*}
By the compactness of the unit ball in $\rr^{d \times d}$, there is a positive constant $C$ such that
$a^\eff_{ijkl}\xi_{ij}\xi_{kl} \ge C|\xi|^2$ for all $\xi \in \rr^{d \times d}$.
\end{proof}

The bilinear form corresponding to \eqref{eq:eff Green operator} takes the form
\begin{align*}
A_\eff(u, v) = \int_{\rr^d} \Big( a_{ijkl}^\eff \partial_{kl} u\partial_{ij} v
+ \bar b_{ij}\partial_j u \partial_i v
+ \frac{1}{2} (H z\cdot z)\, uv + \mu  uv \Big)\, dz,
\end{align*}
and it is coercive.
Namely, there exists a positive constant $\hat C$ such that
\begin{align*}
A_\eff(V, V) \ge \hat C \big(\| \nabla\nabla  V\|_{2, \rr^d}^2
+ \||z|V\|_{2, \rr^d}^2 + \|V\|_{2, \rr^d}^2\big).
\end{align*}
Thus, by the Riesz-Fr\'echet representation theorem, the Green operator is well-defined.
Using Lemma \ref{lm:aeffcoercive} we see that the operator $G$ is self-adjoint.
Moreover, due to the compact embedding of $H^2(\rr^d) \cap L^2(\rr^d, |z|^2dz)$ in $L^2(\rr^d)$, the operator $G$ is compact as an operator in $L^2(\rr^d)$; $L^2(\rr^d, |z|^2dz)$ is a weighted $L^2$-space with the weight $|z|^2$.
As a direct consequence, we have the following result.
\begin{lemma}
\label{lm:spectrum eff prob}
The spectrum of the limit problem
\begin{align}
\label{eq:eff spectral}
\partial_{ij}\big(a_{ijkl}^\eff \partial_{kl} v\big)
- \partial_i \big(\bar b_{ij}\partial_j v\big)
+ \frac{1}{2} (H z\cdot z)\, v = \eta v, \quad z \in \rr^d,
\end{align}
is real, discrete, and consists of a countably infinite number of eigenvalues, each of finite multiplicity:
\begin{align*}
& \eta_1 \le \eta_2 \le \cdots, & \eta_j \to \infty \text{ as } j \to \infty.
\end{align*}
The corresponding eigenfunctions $v_j^\ve$ form an orthonormal basis in
$L^2(\rr^d)$.
\end{lemma}

We proceed to the proof of the convergence of spectra.
We will prove that the Green operator $G^\ve$ converges uniformly to $G$ in $\mathcal{L}(L^2(\rr^d))$.
Then we apply the following result, the proof of the which can be found in \cite[Lemma~2.6]{AlCo} (see also \cite{anselone}),
to conclude the desired convergence of eigenvalues and eigenfunctions.
\begin{lemma}
\label{lm:uniform convergence}
Let $G_\ve$ be a sequence of compact self-adjoint operators acting in $L^2(\rr^d)$. Assume that $G_\ve$ converges uniformly to a compact self-adjoint operator $G$. Let $\eta_k^\ve$ and $\eta_k$ be the $k$th eigenvalues of the operators $G_\ve$ and $G$, respectively; $v_k^\ve$, $v_k$ are eigenfunctions corresponding to $\eta_k^\ve, \eta_k$. Then
as $\ve \to 0$,
\begin{enumerate}[(i)]
\item $\eta_k^\ve \to \eta_k$,
\item up to a subsequence, $v_k^\ve$ converges strongly in $L^2(\rr^d)$ to $v_k$.
\end{enumerate}
\end{lemma}

The uniform convergence of the Green operators is a straightforward consequence of the convergence of the solutions to the corresponding boundary value problems with weakly converging data in $L^2(\rr^d)$, as has been pointed out in \cite[Theorem~2.2]{AlCo}.

\begin{lemma}
\label{lm:pass to the limit}
Let $f_\ve$ be a sequence converging weakly to $f$ in $L^2(\rr^d)$, and let $V^\ve$ be the unique solution of \eqref{eq:Green operator rescaled}. Then $V^\ve$ converges weakly in $H^2(\rr^d)$ and strongly in $L^2(\rr^d)$ to the unique solution $V$ of the effective problem \eqref{eq:eff Green operator}. Moreover,
\begin{align*}
\nabla V^\ve & \overset{2}{\rightharpoonup} \nabla V(z) \,\, \mbox{two-scale in} \,\, L^2(\rr^d),\\
\nabla\nabla V^\ve & \overset{2}{\rightharpoonup} \nabla \nabla V(z) + \nabla_\zeta\nabla_\zeta N_{kl}(\zeta) \partial_{kl} V(z) \text{ two-scale in } L^2(\rr^d),
\end{align*}
where $N_{kl} \in H^2(\ttt^d)/\rr$, $k,l=1, \ldots, d$, solve problem \eqref{eq:N}.
\end{lemma}
\begin{proof}
The proof consists of two parts.
First, we derive a priori estimates for $V^\ve$.
Second, we pass to the two-scale limit in order to obtain the effective problem.

The estimates for $V^\ve$ follows from \eqref{eq:ax-8}:
\begin{align}
\label{eq:a priori est V^eps}
\| \nabla \nabla V^\ve\|_{2, \Omega_\ve}^2
+ \|\nabla V^\ve\|_{2, \Omega_\ve}^2
+ \||z|V^\ve\|_{2, \Omega_\ve}^2 + \|V^\ve\|_{2, \Omega_\ve}^2 \le C\|f_\ve\|_{2, \Omega_\ve}^2.
\end{align}
Having in hand the a priori estimate \eqref{eq:a priori est V^eps}, we deduce (see, for example, Proposition~1.14 in \cite{Al-1992}) that in $L^2(\rr^d)$ we have the following two-scale convergences:
\begin{align}
\label{eq:two-scale conv}
V^\ve & \overset{2}{\rightharpoonup} V(z), \quad
\nabla V^\ve  \overset{2}{\rightharpoonup} \nabla V(z),& 
\nabla\nabla V^\ve & \overset{2}{\rightharpoonup} \nabla \nabla V(z) + \nabla_\zeta\nabla_\zeta W(z,\zeta),
\end{align}
where $W(z, \zeta) \in L^2(\rr^d; H^2(\ttt^d))$.
The strong convergence of $V^\ve$ to $V$ in $L^2(\rr^d)$ follows also from \eqref{eq:a priori est V^eps}, namely from the boundedness of weighted $L^2$-norm, which gives compactness. 
We are going to pass to the limit in the weak formulation of \eqref{eq:Green operator rescaled}:
\begin{align}
& \int_{\Omega_\ve} \hat a_{ijkl}^\ve\partial_{kl}V^\ve \partial_{ij}\Phi^\ve\, dz
+ \int_{\Omega_\ve} \hat b_{ij}^\ve\partial_{i}V^\ve \partial_{j}\Phi^\ve\, dz
+ \frac{1}{\ve^{1/3}}\int_{\Omega_\ve} (\hat c^\ve-\bar c(\ve^{1/6}z)) V^\ve \Phi^\ve\, dz \notag\\
& \quad + \frac{1}{\ve^{1/3}}\int_{\Omega_\ve} (\bar c(\ve^{1/6}z)-\bar c(0)) V^\ve \Phi^\ve\, dz
+ \mu \int_{\Omega_\ve} V^\ve \Phi^\ve\, dz
 = \int_{\Omega_\ve} f_\ve \Phi^\ve\, dz,\label{eq:weak-rescaled-Green}
\end{align}
where $\Phi^\ve(z) = \Phi(z, \frac{z}{\ve^{5/6}})$, $\Phi(z, \zeta) \in C(\overline \Omega_\ve; L^\infty(\ttt^d))$. Note that the term containing $\bar c(\ve^{1/6}z)$ is added and subtracted for convenience, since we are going to use Lemma~\ref{lm:mvt} when passing to the limit. Due to the regularity assumptions (H1), the coefficients can be regarded as a part of a test function.

First we take a test function $\Phi_\ve = \ve^{5/3} \varphi(z)\psi(\frac{z}{\ve^{5/6}})$, with $\varphi \in C_0^\infty(\rr^d), \psi \in C^\infty(\ttt^d)$. Then \eqref{eq:weak-rescaled-Green} transforms into
\begin{align*}
& \int_{\Omega_\ve} \hat a_{ijkl}^\ve\partial_{kl}V^\ve \big(\varphi(z)\partial_{ij}\psi(\zeta) + \ve^{5/6} \partial_j\varphi(z) \partial_i \psi(\zeta)
+ \ve^{5/6} \partial_i\varphi(z) \partial_j \psi(\zeta) \\
& \quad +\ve^{5/3} \partial_{ij}\varphi(z)\psi(\zeta)\big)\big|_{\zeta = \frac{z}{\ve^{5/6}}}\, dz \\
& \quad + \int_{\Omega_\ve} \hat b_{ij}^\ve\partial_{i}V^\ve \big(\ve^{5/6} \varphi(z)\partial_i\psi(\zeta) + \ve^{5/3}\psi(\zeta)\partial_i\varphi(z) \big)\big|_{\zeta = \frac{z}{\ve^{5/6}}}\, dz \\
& \quad +  \ve^{4/3} \int_{\Omega_\ve} (\hat c^\ve-\bar c(\ve^{1/6}z)) V^\ve(z) \varphi(z)\psi(\frac{z}{\ve^{5/6}})\, dz \\
& \quad +  \ve^{4/3}\int_{\Omega_\ve} (\bar c(\ve^{1/6}z)-\bar c(0)) V^\ve(z) \varphi(z)\psi(\frac{z}{\ve^{5/6}})\, dz \\
& \quad + \mu \ve^{5/3} \int_{\Omega_\ve} V^\ve(z)  \varphi(z)\psi(\frac{z}{\ve^{5/6}})\, dz
  = \ve^{5/3} \int_{\Omega_\ve} f_\ve(z) \varphi(z)\psi(\frac{z}{\ve^{5/6}})\, dz.
\end{align*}
Using Lemma~\ref{lm:mvt} and \eqref{eq:two-scale conv} we may pass to the limit, as $\ve \to 0$, using the two-scale convergence, and obtain
\begin{align*}
& \int_{\rr^d}\int_{\ttt^d} a_{ijmn}(0,\zeta)\partial_{\zeta_m \zeta_n}W(z, \zeta) \partial_{ij}\psi(\zeta) \varphi(z)\, d\zeta \, dz \\
& \quad = -\int_{\rr^d}\int_{\ttt^d} a_{ijkl}(0,\zeta)\partial_{kl}V(z) \partial_{ij}\psi(\zeta) \varphi(z)\, d\zeta \, dz.
\end{align*}
From the last identity we deduce that $W(z, \zeta)=N_{kl}(\zeta)\partial_{kl}V(z)$, where the periodic functions $N_{kl}(\zeta)$ solve \eqref{eq:N}.

Now we take $\Phi_\ve = \varphi(z) \in C_0^\infty(\rr^d)$ as a test function in \eqref{eq:weak-rescaled-Green},
and passing to the limit get the weak formulation of the effective problem \eqref{eq:eff Green operator}:
\begin{align*}
& \int_{\rr^d}\int_{\ttt^d} \big(a_{ijmn}(0,\zeta)\partial_{mn}N_{kl}(\zeta) +  a_{ijkl}(0,\zeta)\big) \partial_{ij}\varphi(z)\, d\zeta \, dz \\
& \quad + \int_{\rr^d}\int_{\ttt^d} b_{ij}(0, \zeta)d\zeta \partial_j V(z) \partial_i \varphi(z)\, dz
+\frac{1}{2}\int_{\rr^d} (Hz\cdot z)\, V(z)\varphi(z)\, dz \\
& \quad + \mu \int_{\rr^d} V(z)\varphi(z)\, dz
= \int_{\rr^d} f(z)\varphi(z)\, dz,
\end{align*}
for any $\varphi \in C_0^\infty(\rr^d)$.
Lemma~\ref{lm:pass to the limit} is proved.
\end{proof}

\begin{lemma}\label{lm:greenconvergence}
The Green operator of \eqref{eq:Green operator rescaled} converges uniformly, in $\mathcal{L}(L^2(\rr^d))$, to the Green operator of \eqref{eq:eff Green operator}, as $\ve \to 0$.
\end{lemma}
\begin{proof}
Let $f_\ve \in L^2(\rr^d)$ be a maximizing sequence for $\sup_{ \| f \|_{2,\rr^d} = 1 } \| (G^\ve - G)f \|_{2,\rr^d}$.
By compactness there is a subsequence $f_\ve$ weakly converging to some $f$ in $L^2(\rr^d)$.
By Lemma \ref{lm:pass to the limit}, $G_\ve f_\ve \to Gf$ strongly in $L^2(\rr^d)$, and by the compactness of
$G$, $Gf_\ve \to Gf$ strongly in $L^2(\rr^d)$.
Hence
\begin{align*}
\|G_\ve - G\| \le \|  G_\ve f_\ve - G f  \|_{2,\rr^d} + \| Gf_\ve - Gf  \|_{2,\rr^d} + o(1),
\end{align*}
as $\ve \to 0$, and the convergence along a subsequence follows.
Since the limit $G_\ve f_\ve$ is unique by Lemma \ref{lm:pass to the limit}, the whole sequence converges.
 \end{proof}

Due to the Lemma~\ref{lm:greenconvergence}, the sequence of the Green operators of the rescaled  problem \eqref{eq:Green operator rescaled} converges uniformly to the Green operator of the effective problem \eqref{eq:eff Green operator}. Lemma~\ref{lm:uniform convergence} applied to the Green operators ensures the convergence of spectrum of the rescaled problem \eqref{eq:weak-rescaled-prob}.
\begin{lemma}
\label{lm:conv spectrum rescaled}
Let $(\eta_k^\ve, v_k^\ve)$ be $k$th eigenpair of the rescaled spectral problem \eqref{eq:weak-rescaled-prob}.
Then under the assumptions (H1)--(H4): As $\ve \to 0$,
\begin{enumerate}[(a)]
\item
$\eta_k^\ve \to \eta_k$, where $\eta_k$ is the $k$th eigenvalue of the effective problem \eqref{eq:eff spectral},
\item
 along a subsequence $v_k^\ve$ converges weakly in $H^2(\rr^d)$ and strongly in $L^2(\rr^d)$ to $v_k$,
 where $v_k$ is the eigenfunction corresponding to $\eta_k$ under a proper orthonormalisation.
\end{enumerate}
\end{lemma}
The last lemma combined with \eqref{change of variables} yields Theorem~\ref{th:main}.


\bigskip

In the next sections we consider the cases when $(\alpha, \beta)$ belong to $R_1, R_3, R_4, R_5$.
\section{The case $(\alpha,\beta) \in R_1$}
\label{sec:R_1}

Recall that
\begin{align*}
R_1 & = \{ (\alpha,\beta) : 0 \le \alpha < 1, \,\, 3\alpha < \beta <3\}.
\end{align*}

\begin{theorem}
\label{th:case R_1}
Let $(\alpha, \beta)\in R_1$ and let $(\lambda_k^\ve, u_k^\ve)$ be the $k$th eigenpair of \eqref{eq:orig-prob} normalized by $\|u_k^\ve\|_{2, \Omega}^2=\ve^{d\beta/6}$.
Suppose that the conditions (H1)--(H4) are satisfied.
Then we have the following representation:
\begin{align*}
u_k^\ve = v_k^\ve\big(\frac{x}{\ve^{\beta/6}}\big), \quad
\lambda_k^\ve = \frac{\bar c(0)}{\ve^\beta} + \frac{\eta_k^\ve}{\ve^{2\beta/3}},
\end{align*}
where $(\eta_k^\ve, v_k^\ve)$ are such that as $\ve \to 0$,
\begin{enumerate}[(i)]
\item $\eta_k^\ve \to \eta_k$,
\item up to a subsequence, $v_k^\ve$ converges to $v_k$ weakly in $H^2(\rr^d)$ and strongly in $L^2(\rr^d)$,
\end{enumerate}
where $\eta_k$ is the $k$th eigenvalue, and $v_k$ is an eigenfunction corresponding to $\eta_k$ normalized by $\| v_k \|_{2,\rr^d} = 1$, of the uniformly elliptic effective spectral problem
\begin{align*}
\partial_{ij}(a_{ijkl}^\eff \partial_{kl}v) + \frac{1}{2} (Hz\cdot z) v = \eta v, \quad z \in \rr^d, 
\end{align*}
with $a^\eff_{ijkl}$ defined by \eqref{def:eff coefficients}, and $H = \nabla\nabla \bar c(0)$.
\end{theorem}
\begin{proof}

We shift the spectrum by $\bar c(0)/\ve^\beta$ and make the following the change of variables:
\begin{align*}
\gamma & = \frac{\beta}{6}, &
z & = \frac{x}{\ve^\gamma}, &
v(z) & = u(\ve^\gamma z), &
\eta^\ve & = \ve^{2\beta/3} \big(\lambda^\ve - \frac{\bar c(0)}{\ve^\beta}\big), &
z \in \Omega_\ve = \ve^{-\gamma}\Omega.
\end{align*}
Then we obtain the rescaled problem
\begin{align}\label{eq:case R_1 rescaled}
\begin{cases}
\hat A^\ve_1 v^\ve = \eta^\ve v^\ve, &  z \in \Omega_\ve,\\
v^\ve = \nabla v^\ve \cdot n = 0,    & z \in \partial \Omega_\ve, 
\end{cases}
\end{align}
where
\begin{align}\label{eq:AeffS1}
\hat A^\ve_1 v^\ve & =
\partial_{ij} (\hat a_{ijkl}^\ve \partial_{kl} \, v^\ve )
-
\ve^{ -\alpha + \beta/3 } \partial_i ( \hat b_{ij}^\ve\partial_j \, v^\ve )
+
\ve^{ -\beta/3 } (\hat c^\ve - \bar c(0) ) v^\ve,
\end{align}
and
\begin{align*}
\hat a_{ijkl}^\ve & = a_{ijkl}\big(\ve^{\beta/6}z, \frac{z}{\ve^{1-\beta/6}}\big), &
\hat b_{ij}^\ve & = b_{ij}\big(\ve^{\beta/6}z, \frac{z}{\ve^{1-\beta/6}}\big), &
\hat c^\ve & = c\big(\ve^{\beta/6}z, \frac{z}{\ve^{1-\beta/6}}\big).
\end{align*}
In order to describe the asymptotic behavior of the eigenpairs $(\eta_k^\ve, v_k^\ve)$, as $\ve \to 0$, we prove the uniform convergence of the corresponding Green operators and then use Lemma~\ref{lm:uniform convergence}.

Let $\mu > 0$ and $f_\ve$ be a sequence converging weakly in $L^2(\rr^d)$ to $f$.
Consider the boundary value problem
\begin{align}\label{eq:bvp1}
\begin{cases}
\hat A^\ve_1 V^\ve + \mu V^\ve  = f_\ve, & z \in \Omega_\ve, \\
V^\ve = \nabla V^\ve\cdot n = 0, & z \in \partial \Omega_\ve. 
\end{cases}
\end{align}
By (H1)--(H4), for all sufficiently small $\ve > 0$, the Green operator of \eqref{eq:bvp1} is a compact, self-adjoint and positive operator on $L^2(\rr^d)$.
Moreover, for the sequence of solutions $V^\ve$ to \eqref{eq:bvp1} we have
\begin{align*}
\| \nabla \nabla V^\ve \|_{2,\rr^d}
+
\| \nabla V^\ve \|_{2,\rr^d}
+ \| V^\ve \|_{2,\rr^d}
+ \| |z| V^\ve \|_{2,\rr^d}
& \le C.
\end{align*}
Thus, up to a subsequence,
\begin{align*}
V^\ve & \to V(z), &
\nabla V^\ve & \overset{2}\rightharpoonup \nabla V(z), &
\nabla \nabla V^\ve & \overset{2}\rightharpoonup \nabla \nabla V(z) + \nabla_\zeta \nabla_\zeta W(z, \zeta),
\end{align*}
where $V \in L^2(\rr^d)$, $W(z, \zeta) \in L^2(\rr^d; H^2 (\ttt^d))$.
Passing to the limit in the variational formulation of \eqref{eq:bvp1} we find that $V \in H^2(\rr^d) \cap L^2(\rr^d, |z|^2dz)$ is the unique solution to the equation
\begin{align}\label{eq:bvp1eff}
A^\eff_1 V + \mu V & = f, \quad z \in \rr^d,
\end{align}
where
\begin{align}
\label{eq:case R_1 limit operator}
A^\eff_1 V & = \partial_{ij} (a^\eff_{ijkl} \partial_{kl} V) + \frac{1}{2}(H z\cdot z) V,
\end{align}
and $a_{ijkl}^\eff$ is coercive on $\rr^{d\times d}$ and given by \eqref{def:eff coefficients}. Note that
the second order term vanishes because of the hypothesis $-\alpha + \beta/3 > 0$ and the boundedness of $\nabla V^\ve$. The
Green operator of \eqref{eq:bvp1eff}, as an operator on $L^2(\rr^d)$, is well-defined, is compact, self-adjoint, and positive.
Due to the uniqueness of the solution to \eqref{eq:bvp1eff}, the whole sequence $V^\ve$ converges to $V$.

In this way the Green operator of \eqref{eq:bvp1} converges uniformly to the Green operator of \eqref{eq:bvp1eff}, as $\ve \to 0$. By Lemma~\ref{lm:uniform convergence}, the spectrum of \eqref{eq:case R_1 rescaled} converges to the spectrum of the limit operator \eqref{eq:case R_1 limit operator} in the sense of Kuratowsky convergence of subsets of $\rr$. Changing back the variables yields the desired result.
\end{proof}


\section{The cases $(\alpha,\beta) \in R_3, R_4, R_5$}
\label{sec:R_3-R_5}

Recall that
\begin{align*}
R_3 & = \{ (\alpha, \beta) : 0 < \alpha < 2, \,\, \alpha<\beta<3\alpha, \,\, \beta < \alpha + 2\}, \\
R_4 & = \{ (\alpha, \beta) : \alpha=2, \,\, 2 < \beta < 4\}, \\
R_5 & = \{ (\alpha, \beta) : 2 < \alpha < 4, \,\, \alpha<\beta<4 \}.
\end{align*}
For these regions, the limit problems are of second order and have the same form, but the effective coefficients and the corresponding cell problems are different. Note that we should assume the coerciveness of the matrix $b(x, y)$ so that the effective problems are well-posed.
We gather the results for these cases in the following theorem.

Let the effective coefficients be defined by
\begin{align}
\label{eq:eff coeff R_3-R_5}
b_{ij}^\eff =
\begin{cases}
\int_{\ttt^d} b_{ij}(0,y)\,dy, & (\alpha, \beta)\in R_3,\\
\int_{\ttt^d} b_{ik}(0,y)( \delta_{kj} + \partial_{k}M_j ) \, dy, & (\alpha, \beta)\in R_4,\\
\int_{\ttt^d} b_{ik}(0,y)( \delta_{kj} + \partial_{k}N_j ) \, dy, & (\alpha, \beta)\in R_5.
\end{cases}
\end{align}
where $M_n \in H^2(\ttt^d)/\rr$ and $N_n \in H^1(\ttt^d)/\rr$ are the unique solutions to the respective cell problems
\begin{align}\label{eq:M_j}
\partial_{y_iy_j}( a_{ijkl}(0,y) \partial_{y_ky_l} M_n ) & - \partial_{y_i} ( b_{ij}(0,y) \partial_{y_j} M_n ) \!\!\!\!\!\! & = \partial_{y_i} b_{ni}(0,y), \quad y \in \ttt^d;\\
& - \partial_{y_i} ( b_{ij}(0,y) \partial_{y_j} N_n ) \!\!\!\!\!\! & = \partial_{y_i} b_{ni}(0,y), \quad y \in \ttt^d,\label{eq:N_j}
\end{align}

\begin{theorem}
\label{th:case R_3-R_5}
Let $(\alpha, \beta)\in R_3\cup R_4\cup R_5$ and let $(\lambda_k^\ve, u_k^\ve)$ be the $k$th eigenpair of \eqref{eq:orig-prob} normalized by $\|u_k^\ve\|_{2, \Omega}^2=\ve^{d(\beta-\alpha)/4}$.
Suppose that the conditions (H1)--(H5) are satisfied.
Then the following representation holds:
\begin{align*}
u_k^\ve(x) = v_k^\ve\Big(\frac{x}{\ve^{\frac{\beta-\alpha}{4}}}\Big), \quad
\lambda_k^\ve = \frac{\bar c(0)}{\ve^\beta} + \frac{\eta_k^\ve}{\ve^{\frac{\alpha+\beta}{2}}},
\end{align*}
where $(\eta^\ve_k, v_k^\ve)$ are such that as $\ve \to 0$,
\begin{enumerate}[(i)]
\item $\eta_k^\ve \to \eta_k$,
\item up to a subsequence, $v_k^\ve$ converges to $v_k$ weakly in $H^1(\rr^d)$ and strongly in $L^2(\rr^d)$,
\end{enumerate}
where $\eta_k$ is the $k$th eigenvalue, and $v_k$ is an eigenfunction corresponding to $\eta_k$ normalized by $\| v_k \|_{2,\rr^d} = 1$, of the harmonic oscillator problem
\begin{align}
\label{eq:eff prob R_3-R_5}
-\partial_{i}(b_{ij}^\eff \partial_{j}v) + \frac{1}{2} (Hz\cdot z) v & = \eta v, \quad z \in \rr^d,
\end{align}
with $b^\eff$ defined by \eqref{eq:eff coeff R_3-R_5}, and $H = \nabla\nabla \bar c(0)$.
\end{theorem}
\begin{proof}
We shift the spectrum by $\bar c(0)/\ve^\beta$ and make the following the change of variables:
\begin{align*}
\gamma & = \frac{\beta-\alpha}{4}, &
z & = \frac{x}{\ve^\gamma}, &
v(z) & = u(\ve^\gamma z), &
\eta^\ve & = \ve^{\frac{\alpha+\beta}{2}} \big(\lambda^\ve - \frac{\bar c(0)}{\ve^\beta}\big), &
z \in \Omega_\ve = \ve^{-\gamma}\Omega.
\end{align*}
Then we obtain the rescaled problem
\begin{align}\label{eq:case R_3-R_5 rescaled}
\begin{cases}
\hat A^\ve_2 v^\ve = \eta^\ve v^\ve, & z \in \Omega_\ve,\\
v^\ve = \nabla v^\ve \cdot n = 0, & z \in \partial \Omega_\ve, 
\end{cases}
\end{align}
where
\begin{align}\label{eq:AeffS2}
\hat A^\ve_2 v^\ve & =
\ve^{\alpha-2\gamma}\partial_{ij} (\hat a_{ijkl}^\ve \partial_{kl} \, v^\ve )
-
\partial_i ( \hat b_{ij}^\ve\partial_j \, v^\ve )
+
\ve^{ \alpha-\beta+2\gamma } (\hat c^\ve - \bar c(0) ) v^\ve,
\end{align}
and
\begin{align*}
\hat a_{ijkl}^\ve & = a_{ijkl}\big(\ve^\gamma z, \frac{z}{\ve^{1-\gamma}}\big), &
\hat b_{ij}^\ve & = b_{ij}\big(\ve^\gamma z, \frac{z}{\ve^{1-\gamma}}\big), &
\hat c^\ve & = c\big(\ve^\gamma z, \frac{z}{\ve^{1-\gamma}}\big).
\end{align*}
As above, to describe the asymptotic behavior of the eigenpairs $(\eta_k^\ve, v_k^\ve)$, as $\ve \to 0$, we prove the uniform convergence of the corresponding Green operators and then use Lemma~\ref{lm:uniform convergence}.

Let $\mu > 0$ and $f_\ve$ be a sequence converging weakly in $L^2(\rr^d)$ to $f$.
Consider the boundary value problem
\begin{align}\label{eq:bvp2}
\begin{cases}
\hat A^\ve_2 V^\ve + \mu V^\ve  = f_\ve, & z \in \Omega_\ve, \\
V^\ve = \nabla V^\ve\cdot n = 0, & z \in \partial \Omega_\ve.
 \end{cases}
\end{align}
By (H1)--(H5), for all sufficiently small $\ve > 0$, the Green operator of \eqref{eq:bvp2} is a compact, self-adjoint and positive operator in $L^2(\rr^d)$.
Moreover, for the sequence of solutions $V^\ve$ to \eqref{eq:bvp2} we have
\begin{align*}
\ve^{\frac{3\alpha-\beta}{4}}\| \nabla \nabla V^\ve \|_{2,\rr^d}
+
\| \nabla V^\ve \|_{2,\rr^d}
+ \| V^\ve \|_{2,\rr^d}
+ \| |z| V^\ve \|_{2,\rr^d}
& \le C.
\end{align*}
We proceed by dividing into the cases:
\begin{align*}
\frac{3\alpha-\beta}{4}
\begin{cases}
<1-\gamma, & (\alpha, \beta)\in R_3,\\
=1-\gamma, & (\alpha, \beta)\in R_4,\\
>1-\gamma, & (\alpha, \beta)\in R_5,
\end{cases}
\end{align*}
where $1-\gamma$ will be the scale in the two-scale convergence.

\subsection{$(\alpha, \beta)\in R_3$}
Up to a subsequence,
\begin{align*}
V^\ve & \to V(z), &
\nabla V^\ve & \overset{2}\rightharpoonup \nabla V(z), &
\ve^{\frac{3\alpha-\beta}{4}} \nabla \nabla V^\ve & \overset{2}\rightharpoonup W(z, \zeta),
\end{align*}
where $V \in H^1(\rr^d)$, $W(z, \zeta) \in L^2(\rr^d; H^2 (\ttt^d))$.
Passing to the limit in the variational formulation of \eqref{eq:bvp2} we find that $V \in H^1(\rr^d) \cap L^2(\rr^d, |z|^2dz)$ is the unique solution to the equation
\begin{align}\label{eq:bvp3eff}
-\partial_{i} (b^\eff_{ij} \partial_{j} V) + \frac{1}{2}(H z\cdot z) V + \mu V & = f, \quad z \in \rr^d, 
\end{align}
where $b_{ij}^\eff = \bar b_{ij}(0) = \int_{\ttt^d} b_{ij}(0,y)\,dy$ is coercive on $\rr^d$ by (H5).

\subsection{$(\alpha, \beta)\in R_4$}
In this case $\frac{3\alpha-\beta}{4} = 1 -\gamma$.
Up to a subsequence, 
\begin{align*}
V^\ve & \to V(z), &
\nabla V^\ve & \overset{2}\rightharpoonup \nabla V(z) + \nabla_\zeta W(z, \zeta), &
\ve^{\frac{3\alpha-\beta}{4}} \nabla \nabla V^\ve & \overset{2}\rightharpoonup \nabla_\zeta \nabla_\zeta W(z, \zeta),
\end{align*}
in $L^2(\rr^d)$, where $V \in H^1(\rr^d)$, $W(z, \zeta) \in L^2(\rr^d; H^2 (\ttt^d))$.
Passing to the limit in the variational formulation of \eqref{eq:bvp2} we find that $V \in H^1(\rr^d) \cap L^2(\rr^d, |z|^2dz)$ is the unique solution to the same equation \eqref{eq:bvp3eff} with
\begin{align}
\label{def:beff}
b_{ij}^\eff=\int_{\ttt^d} b_{ik}(0,y)( \delta_{kj} + \partial_{k}M_j ) \, dy,
\end{align}
where $M_j \in H^2(\ttt^d)/\rr$ solves \eqref{eq:N_j}.
By the periodicity of $b(x, y)$, $M_j$ is well-defined.

\begin{lemma}
\label{lm:corciveness b^eff}
Under the assumptions (H1)--(H3) and (H5), $b^\eff$ defined by \eqref{def:beff} is coercive on $\rr^d$.
\end{lemma}
\begin{proof}
By the definition of $b^\eff$,
\begin{align*}
b^\eff_{ij} & = \int_{\ttt^d} \delta_{il} b_{kl}(0,y) ( \delta_{kj} + \partial_{y_k}M_j ) \, dy.
\end{align*}
By using $M_i$ as a test function in the equation \eqref{eq:M_j} for $M_j$ we have
\begin{align*}
\int_{\ttt^d} b_{kl}(0,y)\partial_{y_l} M_i  (\delta_{kj} + \partial_{y_k}M_j) \, dy & = - \int_{\ttt^d} a_{pqrs}(0,y)\partial_{ y_r y_s } M_j \partial_{ y_p y_q } M_i \, dy.
\end{align*}
Therefore,
\begin{align*}
b^\eff_{ij} & = \int_{\ttt^d} b_{kl}(0,y) ( \delta_{li} + \partial_{y_l}M_i )( \delta_{kj} + \partial_{y_k}M_j ) \, dy \\
& \quad + \int_{\ttt^d} a_{pqrs}(0,y)(\partial_{ y_r y_s } M_j) (\partial_{ y_p y_q } M_i) \, dy,
\end{align*}
which shows that $b^\eff$ is symmetric.
Moreover, for $\xi \in \rr^d$, by the last equation, (H3), and (H5), we have
\begin{align*}
b^\eff_{ij}\xi_i \xi_j & \ge C \sum_l \int_{\ttt^d} | \xi_i(\delta_{li} + \partial_{y_l}M_i) |^2 \, dy,
\end{align*}
which shows that $b^\eff$ is nonnegative definite.
If $\xi \in \rr^d$ is such that $b^\eff_{ij}\xi_i \xi_j = 0$ we have by the last inequality that
$|\xi_i(\delta_{li} + \partial_{y_l}M_i)| = 0$ for all $l$.
In particular, $\nabla_y ( \xi_i( y_i + M_i ) ) = 0$ which by the periodicity of $M_i$ is only possible if $\xi = 0$.
Thus $b^\eff$ is positive definite on $\rr^d$.
By the compactness of the unit ball in $\rr^d$,
$b^\eff$ is coercive on $\rr^d$.
\end{proof}

\subsection{$(\alpha, \beta)\in R_5$}
Up to a subsequence,
\begin{align*}
V^\ve & \to V(z), &
\nabla V^\ve & \overset{2}\rightharpoonup \nabla V(z) + \nabla_\zeta N_j(z, \zeta)\partial_j V(z), &
\ve^{\frac{3\alpha-\beta}{4}} \nabla \nabla V^\ve & \overset{2}\rightharpoonup W_1(z, \zeta),
\end{align*}
in $L^2(\rr^d)$, where $V \in H^1(\rr^d)$, $W_1(z, \zeta) \in L^2(\rr^d \times \ttt^d)$, and $N_j \in H^1(\ttt^d)/\rr$ solves \eqref{eq:N_j}.
Passing to the limit in the variational formulation of \eqref{eq:bvp2} we find that $V \in H^1(\rr^d) \cap L^2(\rr^d, |z|^2dz)$ is the unique solution to \eqref{eq:bvp3eff} with
\begin{align*}
b_{ij}^\eff = \int_{\ttt^d} b_{ik}(0,y)( \delta_{kj} + \partial_{k}N_j ) \, dy.
\end{align*}
By the similar argument used in Lemma~\ref{lm:corciveness b^eff}, $b^\eff$ is coercive on $\rr^d$.

\bigskip

In this way, in all the three cases, the Green operator of \eqref{eq:bvp2} converges uniformly to the Green operator of \eqref{eq:bvp3eff}, as $\ve \to 0$. By Lemma~\ref{lm:uniform convergence}, the spectrum of \eqref{eq:case R_3-R_5 rescaled} converges to the spectrum of the limit operator \eqref{eq:eff prob R_3-R_5}. Changing back the variables yields the desired result.
\end{proof}


\section*{Acknowledgements}

This work was done during the stay of A. Chechkina at Narvik University College in 2014,
the hospitality of which is kindly acknowledged.

\bibliographystyle{plain}
\bibliography{refs}

\end{document}